\documentclass[10pt,a4paper,oneside]{amsart}

\usepackage[T1]{fontenc}
\usepackage[margin=1in]{geometry}

\usepackage{amsmath,amsfonts,amssymb,amsthm,url,verbatim}
\usepackage[dvips]{graphicx}
\usepackage[croatian,english]{babel}
\usepackage{amsfonts}
\usepackage[utf8]{inputenc}
\usepackage[T1]{fontenc}
\usepackage{amsmath}
\usepackage{array}
\usepackage{amssymb}
\usepackage[thinlines]{easytable}
\usepackage{ctable}
\usepackage{enumitem}
\usepackage{makecell}
\usepackage{blindtext}
\usepackage{tabularx}
\usepackage{mathtools}
\usepackage{fancyref}
\usepackage{arydshln}
\usepackage{latexsym}
\usepackage[dash,dot]{dashundergaps}
\usepackage{soul}
\usepackage[pagebackref]{hyperref}
\usepackage{xcolor}
\usepackage{cleveref}

\definecolor{orange}{RGB}{209, 103, 17}
\definecolor{OliveGreen}{RGB}{4, 110, 15}

\DeclareMathAlphabet{\mathbfit}{OML}{cmm}{b}{it}

\usepackage{tabularx}
\makeatletter
\def\hlinewd#1{%
	\noalign{\ifnum0=`}\fi\hrule \@height #1 %
	\futurelet\reserved@a\@xhline}
\makeatother
\newtheorem{theorem}{Theorem}[section]

\newtheorem{proposition}[theorem]{Proposition}
\newtheorem{lemma}[theorem]{Lemma}

\newtheorem{question}{Question}
\theoremstyle{definition}
\newtheorem{remark}[theorem]{Remark}
\newtheorem{definition}[theorem]{Definition}

\newcommand{\Q}{\mathbb{Q}}

\newcommand{\Z}{\mathbb{Z}}
\newcommand{\F}{\mathbb{F}}

\newcommand{\PP}{\mathfrak{P}}
\newcommand{\OO}{\mathcal{O}}
\newcommand{\Gal}{\operatorname{Gal}}

\newcommand{\fp}{\mathfrak p}
\newcommand{\fq}{\mathfrak q}
\renewcommand{\P}{\mathbb{P}}
\DeclareMathOperator{\res}{res}

\DeclareMathOperator{\Frob}{Frob}

\title{Splitting of primes in number fields generated by points on some modular curves}
\date{\today}
\author{Filip Najman}
\address{Department of Mathematics, Faculty of Science, University of Zagreb, Bijenička cesta 30, 10000 Zagreb, Croatia}
\email{fnajman@math.hr}
\author{Antonela Trbović}
\address{Department of Mathematics, Faculty of Science, University of Zagreb, Bijenička cesta 30, 10000 Zagreb, Croatia}
\email{antonela.trbovic@math.hr}
\thanks{The authors were supported by the QuantiXLie Centre of Excellence, a project
	co-financed by the Croatian Government and European Union through the
	European Regional Development Fund - the Competitiveness and Cohesion
	Operational Programme (Grant KK.01.1.1.01.0004) and by the Croatian Science Foundation under the project no. IP-2018-01-1313.}

\begin{document}
\maketitle

\begin{abstract}
	We study the splitting of primes in number fields generated by points on modular curves. Momose \cite{momose} was the first to notice that quadratic points on $X_1(N)$ generate quadratic fields over which certain primes split in a particular way and his results were later expanded upon by Krumm \cite{kru}. We prove results about the splitting behaviour of primes in quadratic fields generated by points on the modular curves $X_0(N)$ which are hyperelliptic (except for $N=37$) and in cubic fields generated by points on $X_1(2,14)$.
\end{abstract}

\section{Introduction}

A famous and much-studied problem in the theory of elliptic curves, going back to Mazur's torsion theorem \cite{mazur}, is to determine the possible torsion groups of elliptic curves over $K$, for a given number field $K$ or over all number fields of degree $d$. Here we are more interested in the inverse question:
\begin{question}
\label{ques}

For a given torsion group $T$ and a positive integer $d$, for which and what kind of number fields $K$ of degree $d$ do there exist elliptic curves $E$ such that $E(K)\simeq T$?

 \end{question}

 To make \Cref{ques} sensible, one should of course choose the group $T$ in a such a way that the set of such fields  should be non-empty and preferably infinite.

It has been noted already by Momose \cite{momose} in 1984 (see also \cite{km}) that the existence of specific torsion groups $T$ over a quadratic field $K$ forces certain rational primes to split in a particular way in $K$. Krumm \cite{kru} in his PhD thesis obtained similar results about splitting of primes over quadratic fields $K$ with $T\simeq \Z/13\Z$ or $\Z/18 \Z$ and it was also proven by Bosman, Bruin, Dujella and Najman \cite{bbdn} and Krumm \cite{kru} independently that all such quadratic fields must be real.

The first such result over cubic fields was proven by Bruin and Najman \cite{BN2}, where it was shown for $T\simeq \Z/2\Z \times \Z/14\Z$ that all such cubic fields $K$ must be cyclic. In this paper we explore this particular case further and prove in \Cref{sec:cubic} that in such a field $2$ always splits, giving the first description of a splitting behaviour forced by the existence of a torsion group of an elliptic curve over a cubic field. Furthermore, we show that all primes $q\equiv \pm 1 \pmod 7$ of multiplicative reduction for such curves split in $K$. The proof of these results turns out to be more intricate than in the quadratic case.

As \Cref{ques} can equivalently be phrased as asking when the modular curve $X_1(M,N)$ parameterizing elliptic curves together with the generators of a torsion subgroup $T\simeq \Z/M\Z \times \Z/N\Z$ has non-cuspidal points over $K$, one is naturally drawn to ask a more general question by replacing $X_1(M,N)$ by any modular curve $X$.

\begin{question}
\label{ques2}
For a given modular curve $X$ and a positive integer $d$, for which and what kind of number fields $K$ of degree $d$ do there exist non-cuspidal points in $X(K)$?
 \end{question}

The most natural modular curves to consider next are the classical modular curves $X_0(N)$ classifying elliptic curves with cyclic isogenies of degree $N$. Bruin and Najman \cite{BN} studied quadratic points on the modular curves $X_0(N)$ which are hyperelliptic and showed that for $N\neq 37$ all quadratic points, except finitely many explicitly listed exceptions, are \textit{non-exceptional}, meaning that if $X_0(N)$ is written as $y^2=f_N(x)$, then $x$ is $\Q$-rational (see also \Cref{exceptional}). Furthermore, they show that the non-exceptional points correspond to $\Q$-curves, i.e. elliptic curves that are isogenous to their Galois conjugates (see \cite[Section 2.1]{BN} for an explanation of this fact). They also proved that quadratic fields $K$ over which $X_0(N)$ for $N=28$ and $40$ have non-cuspidal points are always real. In this paper we prove the first results about splitting of certain primes over quadratic fields where some modular curves $X_0(N)$ have non-cuspidal points. We consider all the $N$ such that $X_0(N)$ is hyperelliptic except for $N=37$, in particular
\begin{equation}\label{lista}
  N\in \{22, 23, 26, 28, 29, 30, 31, 33, 35, 39, 40, 41, 46, 47, 48, 50, 59, 71\}.
\end{equation}
The reason we exclude $N=37$ is that the quadratic points on $X_0(37)$ cannot all be described (with finitely many exceptions) as inverse images of $\P^1(\Q)$ with respect to the degree $2$ hyperelliptic map $X_0(37)\rightarrow \P^1$. For more details about quadratic points on $X_0(37)$, see \cite{box}. In \Cref{sec:quadratic} we prove a series of results about the splitting behaviour of various primes in quadratic fields generated by quadratic points on $X_0(N)$.

 A difficulty in proving these results that one immediately encounters is that the methods of \cite{km} and \cite{momose} cannot be adapted to $X_0(N)$ as the existence of a torsion point of large order forces bad reduction on the elliptic curve (see for example \cite[Lemma 1.9]{momose}), while the existence of an isogeny does not. Hence we approach the problem via explicit equations and parameterizations of modular curves, more in the spirit of \cite{bbdn,BN,kru} instead of moduli-theoretic considerations as in \cite{km, momose}.

Our results have some overlap with prior work of Gonz\'alez \cite{gonz}, Clark \cite{clark} and Ozman \cite{EO}. Gonz\'alez proved results about fields generated by $j$-invariants of $\Q$-curves. Since, as mentioned before, for the values $N$ that we study almost all $N$-isogenies over quadratic fields come from $\Q$-curves, our results are reminiscent of his, but it turns out there is little overlap in the results that are proved. This is perhaps not very surprising as we do not use the fact that we are looking at $\Q$-curves at all. Ozman and Clark, on the other hand, give criteria for twists of the modular curves $X_0(N)$ to not have local points. Using our approach we reprove some of their results. More details on the overlap of our and their results can be found in \Cref{rem:ozman}.

Our main results over quadratic fields are \Cref{thm2.1a}, \Cref{thm2.1b} and \Cref{thm2.1c}. We do not state them here in the introduction as the statements are quite long. Writing $X_0(N)$ as $y^2=f_N(x)$, we prove these results by showing how various properties of $f_N$ determine the possible values of $f_N(x)$ modulo squares, for $x\in \Q$, and how this restricts the possible quadratic fields over which $X_0(N)$ might have non-cuspidal points. To do this, we use mostly known results from algebraic number theory, although not always in a straightforward way.

Before we proved these results, we ran extensive numerical computations where we looked for patterns in the quadratic fields over which hyperelliptic curves $X_0(N)$ have non-cuspidal points. This is how we obtained most of the results that we wanted to prove, and which we indeed proved afterwards. Our results are exhaustive in the sense that we have explained and proved all the patterns that we have noticed, although of course we cannot exclude the possibility that there exist other results of this type that we have overlooked.

The computations in this paper were executed in the computer algebra system
Magma \cite{MAGMA}. Throughout the paper, in places where we used Magma to compute something, we use the phrase "we compute" or something very similar. The code used in this paper can be found at
\url{https://web.math.pmf.unizg.hr/~atrbovi/magma/magma2.htm}.

\section{Splitting of primes in quadratic fields generated by points on $X_0(N)$}
\label{sec:quadratic}

In this section we study the splitting behaviour of primes in quadratic fields over which the modular curves $X_0(N)$ have non-cuspidal points. Models for $X_0(N)$ have been obtained from the \texttt{SmallModularCurves} database in Magma and can be found in Table \ref{tab:jed}.

\newpage
\begin{table}[h]
	\begin{tabular}{cl}\hlinewd{0.8pt}	\\[-0.8em]
		\multicolumn{1}{c}{$\mathbfit{N} $} & \multicolumn{1}{l}{\textup{\textbf{$\mathbfit{f_N(x)}$ from the equation $\mathbfit{y^2=f_N(x)}$ for $\mathbfit{X_0(N)}$ and the factorization in $\mathbfit{\Q[X]}$}}} \\ \hlinewd{0.8pt}
		\\[-0.8em]
		$\:\:\:\:\mathbf{22}\:\:\:\:$                        & $x^6 - 4x^4 + 20x^3 - 40x^2 + 48x-32$\\ &$=(x^3-2x^2 + 4x - 4)(x^3 + 2x^2 - 4x + 8)$          \\ \hdashline
		\\[-0.8em]
		$\mathbf{23}$     &   $x^6-8x^5+2x^4+2x^3-11x^2+10x-7$	\\ &$=(x^3 - 8x^2 + 3x - 7)(x^3 - x + 1)$         							                                              \\ \hdashline
		\\[-0.8em]
		$\mathbf{26}$                        & $x^6-8x^5+8x^4-18x^3+8x^2-8x+1$     							                                                                   \\ \hdashline
		\\[-0.8em]
		$\mathbf{28}$                        & $4x^6-12x^5+25x^4-30x^3+25x^2-12x+4	$  	\\ &$=(2x^2 - 3x + 2)(x^2 - x + 2)(2x^2 - x + 1)$         							                                                             \\ \hdashline
		\\[-0.8em]
		$\mathbf{29}$                        & $x^6 - 4x^5 - 12x^4 + 2x^3 + 8x^2 + 8x-7	$                                              \\ \hdashline
		\\[-0.8em]
		$\mathbf{30}$                        & $x^8 + 14x^7 + 79x^6 + 242x^5 + 441x^4 +
		484x^3 + 316x^2 + 112x + 16$     \\ &$=(x^2 + 3x + 1)(x^2 + 6x + 4)(x^4 + 5x^3 + 11x^2 + 10x + 4)$         							                                                                                          \\ \hdashline
		\\[-0.8em]
		$\mathbf{31}$                        & $x^6 -8x^5 + 6x^4 + 18x^3 -11x^2 -
		14x -3						$      	\\ &$=(x^3 - 6x^2 - 5x - 1)(x^3 - 2x^2 - x + 3)$         							                                                                                 \\ \hdashline
		\\[-0.8em]
		$\mathbf{33}$                        & $ x^8 + 10x^6 - 8x^5 + 47x^4 - 40x^3 +
		82x^2 - 44x + 33	$         	\\ &$=(x^2 - x + 3)(x^6 + x^5 + 8x^4 - 3x^3 + 20x^2 - 11x + 11)$         					                                         \\ \hdashline
		\\[-0.8em]
		$\mathbf{35}$                        & $ x^8 - 4x^7 - 6x^6 - 4x^5 - 9x^4 + 4x^3
		- 6x^2 + 4x + 1	$  	\\ &$=(x^2 + x - 1)(x^6 - 5x^5 - 9x^3 - 5x - 1)$         					                                                        \\ \hdashline
		\\[-0.8em]
		$\mathbf{39}$                        & $ x^8 - 6x^7 + 3x^6 + 12x^5 -23x^4 +
		12x^3 + 3x^2 - 6x + 1	$   \\ &$=(x^4 - 7x^3 + 11x^2 - 7x + 1)(x^4 + x^3 - x^2 + x + 1)$         					                                                  \\ \hdashline
		\\[-0.8em]
		$\mathbf{40}$                        & $ x^8 + 8x^6 - 2x^4 + 8x^2 + 1$						                                              \\ \hdashline
		\\[-0.8em]
		$\mathbf{41}$                        & $ x^8 -4x^7- 8x^6 + 10x^5 + 20x^4 +
		8x^3 -15x^2 -20x - 8$                                              \\ \hdashline
		\\[-0.8em]
		$\mathbf{46}$                        & $x^{12} - 2x^{11} + 5-x^{10} + 6x^9 - 26x^8 +
		84x^7 - 113x^6 + 134x^5 - 64x^4 + 26x^3 + 12x^2 + 8x - 7							$     \\ &$=(x^3 - 2x^2 + 3x - 1)(x^3 + x^2 - x + 7)(x^6 - x^5 + 4x^4 - x^3 + 2x^2 + 2x + 1)$         				                                         \\ \hdashline
		\\[-0.8em]
		$\mathbf{47}$                        & $x^{10}-6x^9 + 11x^8 - 24x^7 + 19x^6 -
		16x^5 - 13x^4 + 30x^3 - 38x^2 + 28x - 11	$         \\ &$=(x^5 - 5x^4 + 5x^3 - 15x^2 + 6x - 11)(x^5 - x^4 + x^3 + x^2 - 2x + 1)$         				                                                                             \\ \hdashline
		\\[-0.8em]
		$\mathbf{48}$                        & $x^8 + 14x^4 + 1	$         \\ &$=(x^4 - 2x^3 + 2x^2 + 2x + 1)(x^4 + 2x^3 + 2x^2 - 2x + 1)$         				                                                               \\ \hdashline
		\\[-0.8em]
		$\mathbf{50}$                        & $x^6-4x^5 -10x^3 - 4x + 1	$                                              \\ \hdashline
		\\[-0.8em]
		$\mathbf{59}$                        & $x^{12} - 8x^{11} + 22x^{10} - 28x^9 + 3x^8 +
		40x^7 - 62x^6 + 40x^5 - 3x^4 - 24x^3 + 20x^2 - 4x - 8					$                                       \\ &$=(x^3 - x^2 - x + 2)(x^9 - 7x^8 + 16x^7 - 21x^6 + 12x^5 - x^4 - 9x^3 + 6x^2 - 4x - 4)$                 \\ \hdashline
		\\[-0.8em]
		$\mathbf{71}$                        & $x^{14} + 4x^{13} - 2x^{12} - 38x^{11} - 77x^{10}
		- 26x^9 + 111x^8 + 148x^7 +$  \\
		      & \hspace{0.8cm}$	+ x^6 - 122x^5 - 70x^4 + 30x^3 + 40x^2 +
		4x - 11$  \\ &$=(x^7 - 7x^5 - 11x^4 + 5x^3 + 18x^2 + 4x - 11)(x^7 + 4x^6 + 5x^5 + x^4 - 3x^3 - 2x^2 + 1)$  \\
			
			\hlinewd{0.8pt}
	\end{tabular}
	\vspace{0.1cm}
	\caption{Polynomials $f_N(x)$ in the equations $y^2=f_N(x)$ for $X_0(N).$}
	\label{tab:jed}
\end{table}


\begin{definition} \label{exceptional}
On a hyperlliptic curve $X$ with a model $y^2=f(x)$, we say that the quadratic points on $X$ of the form $(x_0, \sqrt{f(x_0)})$, where $x_0\in \Q$, are \textit{non-exceptional}. The quadratic points that are not of that form are called \textit{exceptional}.
\end{definition}

By the results of \cite{BN}, all non-cuspidal quadratic points on $X_0(N)$ are non-exceptional, with finitely many explicitly listed exceptions.

\setul{3.5pt}{.4pt}

\newpage
\begin{theorem}\label{thm2.1a}
	Let $K=\Q(\sqrt{D})$, where $D$ is squarefree, be a quadratic field over which $X_0(N)$ has an non-exceptional non-cuspidal point.
For each $N$, the table below shows the splitting behaviour of some of the small primes in $K$, as well as some properties of $D$.

	\begin{table}[h]
	\begin{tabular}{clllc}\hlinewd{0.8pt}\\[-0.8em]
		\multicolumn{1}{c}{$\mathbfit{N} $} & \multicolumn{1}{c}{\textup{\textbf{not inert}}}&\multicolumn{1}{c}{\textup{\textbf{unramified}}}&\multicolumn{1}{c}{\textup{\textbf{splits}}}&\multicolumn{1}{r}{\textup{\textbf{$\:\:\:\:\: \mathbfit{D}\:\:\:\:\:$}}} \\ \hlinewd{0.8pt}\\[-0.8em]
		$\:\:\:\:\:\:\:\:\mathbf{22}\:\:\:\:\:\:\:\:$                       & $2^{*}$  &&         \\ \hdashline\\[-0.8em]
		$\mathbf{26}$                        &$13$&&& \textup{odd}    \\ \hdashline\\[-0.8em]
		$\mathbf{28}$                        &$3,7$&$3$&$3$&$>0$             \\ \hdashline\\[-0.8em]
		$\mathbf{29}$                        &$29$&&& \textup{odd}      \\ \hdashline\\[-0.8em]
		$\mathbf{30}$                        &$2,3,$ $5^{**}$&$2,3$&$2,3$&\textup{odd}      \\ \hdashline\\[-0.8em]
		$\mathbf{33}$                        &$2,11$&$2$& $2$& \thead{$>0$\\ \textup{odd} }   \\ \hdashline\\[-0.8em]
		$\mathbf{35}$                        & $5^{**},7$&$2,7$&$7$&\textup{odd}   \\ \hdashline\\[-0.8em]
		$\mathbf{39}$                        &$3,13$&$2,13$&$13$ &\textup{odd}          \\ \hdashline\\[-0.8em]
		$\mathbf{40}$                        &$2,3,5$&$2,3,5$&$2,3,5$& \thead{$>0$\\ \textup{odd} }        \\ \hdashline\\[-0.8em]
		$\mathbf{41}$                        &$41$&&  & \\ \hdashline\\[-0.8em]
		$\mathbf{46}$                        &$2$&$2$&$2$&\textup{odd}     \\ \hdashline\\[-0.8em]
		$\mathbf{48}$                        &$2$&$2,3,5$&$2,3,5$&\thead{$>0$\\ \textup{odd} }	   \\ \hdashline\\[-0.8em]
		$\mathbf{50}$                        &$5$&&& \textup{odd}     \\ \hlinewd{0.8pt}\\[-0.8em]
		\multicolumn{5}{l}{\hspace{0.6cm}$^{*}$ \textup{-even more is true,} $D\equiv 1,2,6\pmod 8$}  \\
		\multicolumn{5}{l}{\hspace{0.6cm}$^{**}$ \textup{-even more is true,} $D\equiv 0,1\pmod 5$}  \\
		\hlinewd{0.8pt}
	\end{tabular}
	\vspace{0.1cm}
	\caption{}
	
\end{table}\label{table:ram}
\end{theorem}

The proof will be similar for different values of $N,$ so before proceeding to a case-by-case study, we mention some general results which will be useful.

We fix the following notation throughout this section. Let $N$ be one of the integers from \eqref{lista} and write
$$X_0(N):y^2=f_N(x)=\sum_{i=0}^{\deg f_N}a_{i,N}x^{i},$$ $\text{with }a_{i,N}\in \Z.$ Note that in all instances $\deg f_N$ is even. As already stated, all non-cuspidal quadratic points on $X_0(N)$ are non-exceptional, with finitely many exceptions. Those exceptions can be found listed in \cite[Tables 1-18]{BN}. Let $(x_0, \sqrt{f_N(x_0)})$, for some $x_0 \in \Q,$ be a non-exceptional point on $X_0(N)$ and write $x_0=m/n,$ with $m$ and $n$ coprime integers. Let $d:=f_N(x_0)$, $s:=n^{\deg f_N}d$, and let $D$ be the square-free part of $d$, i.e. the unique square-free integer such that $n^{\deg f_N}d=Ds^2$, for some $s\in \Q.$ Since $\deg f_N$ is even, it follows that $s\in \Z$.
We get the equality
\begin{equation}\label{jed:fakt}
n^{\deg f_N}d=Ds^2=\sum_{i=0}^{\deg f_N}a_{i,N}m^in^{deg f_N-i}.
\end{equation}
The point $(x_0, \sqrt{f_N(x_0)})$ will be defined over $K:=\Q(\sqrt D)$.

Before proceeding with the proof of Theorem \ref{thm2.1a} we mention a number of lemmas that describe the splitting behaviour of primes in $K$, which will be used throughout. They are well-known or obvious, so we omit the proofs.

\begin{lemma}\label{lem:nep1}
	An odd prime $p$ ramifies in $K$ if and only if $p\mid D$, splits in $K$ if and only if $\left(\frac{D}{p}\right)=1$ and is inert in $K$ if and only if $\left(\frac{D}{p}\right)=-1$.
\end{lemma}

\begin{lemma}\label{lem:nep2}
	Let $p$ be an odd prime and assume that we have $Ds^2\equiv a p^t \pmod{p^\ell}$ with $p\nmid a$ and $\ell>t$.
	\begin{itemize}
		\item[a)] If $t=2k$ for some $k\in \Z_0^+$, then $v_p(s)=k$, $D\equiv a (p^k/s)^2\pmod{p^{\ell-t}}$, and $p$ splits in $K$ if and only if $a$ is a square modulo $p$.
		\item[b)] If $t=2k+1$ for some $k\in \Z_0^+$, then $p|D$ and $p$ ramifies in $K$.
	\end{itemize}
\end{lemma}

\noindent As previous lemmas stated results about splitting for odd primes, we include similar results for the prime $p=2.$

\begin{lemma}\label{lem:par1}
	The prime $2$ ramifies in $K$ if and only if $D\not\equiv 1 \pmod 4$, splits in $K$ if and only if $D\equiv 1 \pmod 8$ and is inert in $K$ if and only if $D\equiv 5 \pmod 8$.
\end{lemma}

\begin{lemma}\label{lem:par2}
	Assume that we have $Ds^2\equiv 2^{t}a \pmod{2^{\ell}}$, with $2\nmid a$ and $\ell>t$.
	\begin{itemize}
		\item[(a)] If $t=2k,$ for some $k\in \Z_0^+$, then $v_2(s)=k$ and $D\equiv a(2^k/s)^2\pmod{2^{\ell -t}}.$ If $a=1$ and $\ell -t=3$, then $2$ splits in $K$.
		\item[(b)] If $t=2k+1,$ for some $k\in \Z_0^+$, then $D\equiv 2a\pmod{2^{\ell -2k}}.$
	\end{itemize}
	
\end{lemma}

\noindent All of the computations done in the following proof are listed in the accompanying \href{https://web.math.pmf.unizg.hr/~atrbovi/magma/magma2.htm}{Magma code}.\\

\begin{proof}[Proof of \Cref{thm2.1a}]

\noindent$\mathbfit{N}\mathbf{=22:}$
In the manner described above, in \eqref{jed:fakt} we get
\begin{equation*}
n^{6}d=Ds^2=m^6-4m^4n^2+20m^3n^3-40m^2n^4+48mn^5-32n^6.
\end{equation*}

Considering all of the possibilities of $m$ and $n$ modulo 512, with the accompanying code we compute that $Ds^2\equiv 1\pmod{8},$ $Ds^2\equiv 32\pmod{64}$ or $Ds^2\equiv 64\pmod{512}.$ Using \Cref{lem:par2} this becomes $D\equiv 1\pmod{8}$ or $D\equiv 2\pmod{4}.$ In any case we have $D\equiv 1,2,6\pmod{8}$, so 2 is not inert, according to \Cref{lem:par1}.
\newline

\noindent$\mathbfit{N}\mathbf{=26:}$
In \eqref{jed:fakt} we get
\begin{equation*} n^6d=Ds^2=m^6-8m^5n+8m^4n^2-18m^3n^3+8m^2n^4-8mn^5+n^6.
\end{equation*}

Looking at all the possibilities of $m$ and $n$ modulo $13^2$, we compute that $Ds^2\equiv 1,3,4,9,10,12\pmod{13}$ or $Ds^2\equiv 4\cdot 13,9\cdot 13\pmod{13^2}$. It follows from \Cref{lem:nep2} that $D\equiv 0, 1, 3, 4, 9, 10, 12\pmod{13}.$ Using \Cref{lem:nep1} we immediately get that $13$ is not inert.

Considering the possibilities of $m$ and $n$ modulo 128, we compute that $Ds^2\equiv 1\pmod{2},$ $Ds^2\equiv 4\pmod{16},$ $Ds^2\equiv 16\pmod{32}$ or $Ds^2\equiv 64\pmod{128}.$ Using \Cref{lem:par2} this becomes $D\equiv 1\pmod{2}$ or $D\equiv 1\pmod{4}$, so $D$ is always odd.
\newline

\noindent$\mathbfit{N}\mathbf{=28:}$		
In \eqref{jed:fakt} we get
\begin{equation*} n^6d=Ds^2=4m^6-12m^5n+25m^4n^2-30m^3n^3+25m^2n^4-12mn^5+4n^6.
\end{equation*}

Considering the possibilities of $m$ and $n$ modulo 3, we compute $Ds^2\equiv 1\pmod 3,$ so from \Cref{lem:nep2} we have $D\equiv 1\pmod{3},$ and the fact that $3$ splits follows from \Cref{lem:nep1}.

Looking at all the possibilities of $m$ and $n$ modulo $7^2$, we compute that $Ds^2\equiv 1,2,4\pmod{7}$ or $Ds^2\equiv 14\pmod{7^2}$. It follows from \Cref{lem:nep2} that $D\equiv 0, 1, 2, 4\pmod{7}$ and from from \Cref{lem:nep1} that $7$ is not inert.

The proof of the fact that $D>0$ can be found in \cite[Theorem 4]{BN}.
\newline

\noindent$\mathbfit{N}\mathbf{=29:}$		
In \eqref{jed:fakt} we get
\begin{equation*} n^6d=Ds^2=m^6-4m^5n-12m^4n^2+2m^3n^3+8m^2n^4+8mn^5-7n^6.
\end{equation*}

Considering the possibilities of $m$ and $n$ modulo 32, we compute that $Ds^2\equiv 1\pmod{2},$  $Ds^2\equiv 12\pmod{16}$ or $Ds^2\equiv 16\pmod{32}.$ Using \Cref{lem:par2} this becomes $D\equiv 1\pmod{2}$ or $D\equiv 3\pmod{4}$, so $D$ is always odd.

We write $D=29^a\cdot p_1\cdot ... \cdot p_k, $ where $a\in\{0,1\}$ and $p_i\neq 2,$ since $D$ is odd. If $a=1,$ then $D\equiv 0\pmod{29}.$ If $a=0,$ then $\left( \frac{D}{29} \right)=\left( \frac{p_1}{29} \right)\cdot ... \cdot\left( \frac{p_k}{29} \right),$ which is equal to 1 after using \Cref{thm2.1b} for $N=29$, which was proved independently. In this case we have that $\left( \frac{D}{29} \right)=1,$ and \Cref{lem:nep1} says that $29$ is not inert.
\newline

\noindent$\mathbfit{N}\mathbf{=30:}$		
In \eqref{jed:fakt} we get
\begin{equation*} n^8d=Ds^2=m^8+14m^7n+79m^6n^2+242m^5n^3+441m^4n^4+484m^3n^5+316m^2n^6+112mn^7+16n^8.
\end{equation*}

Considering the possibilities of $m$ and $n$ modulo 128, we compute that $Ds^2\equiv 16\pmod{128}$ or $Ds^2\equiv 1\pmod{8}.$ Using \Cref{lem:par2} we get $D\equiv1\pmod{8},$ and from \Cref{lem:par1} we conclude that 2 splits.

Considering the possibilities of $m$ and $n$ modulo 3, we compute that $Ds^2\equiv 1 \pmod{3},$ and from \Cref{lem:nep2} we conclude that $D\equiv1\pmod{3}.$ The fact that 3 splits follows from \Cref{lem:nep1}.

Looking at all the possibilities of $m$ and $n$ modulo 25, we compute that $Ds^2\equiv 1\pmod{5}$ or $Ds^2\equiv 5\pmod{25}$. Using \Cref{lem:nep2} we get $D\equiv 0,1,4\pmod{5}$ and from \Cref{lem:nep1} we see that 5 is not inert.

\noindent Furthermore, we want to eliminate the possibility $D\equiv 4\pmod{5}.$ If it were true, then for $s$ in $n^8d=Ds^2$ it holds $s^2\equiv 4\pmod 5$, so $s$ would be divisible by a prime $p$ such that $p\equiv 2,3 \pmod 5$, i.e. $\left(\frac{5}{p}\right)=-1.$

The expression $n^8d=Ds^2$ above factorizes as
\begin{equation*}
n^8d=Ds^2=\left(m^2+6nm+4n^2\right)\left(m^2+3nm+n^2\right)\left(m^4+5m^3n+11m^2n^2+10mn^3+4n^4\right),
\end{equation*}
so $p$ has to divide one of the 3 factors on the right.
\begin{itemize}
	\item[$\bullet$] If $p$ divides $m^2+6nm+4n^2=(m+3n)^2-5n^2,$ then $\left(\frac{5}{p}\right)=1$, so $p\not\equiv 2,3 \pmod 5.$
	\item[$\bullet$] If $p$ divides the second factor, it also divides $4(m^2+3nm+n^2)=(2m+3n)^2-5n^2,$ then $\left(\frac{5}{p}\right)=1$, so $p\not\equiv 2,3 \pmod 5.$
	\item[$\bullet$] If $p$ divides $m^4+5m^3n+11m^2n^2+10mn^3+4n^4 = (2m^2+5mn+4n^2)^2 + 3m^2n^2,$ then $\left(\frac{-3}{p}\right)=1$. The third factor can also be written as $(2m^2+5mn+m^2)^2 + 15(n^2+mn)^2,$ so we also have $\left(\frac{-15}{p}\right)=1.$ Combining these two facts, we get $\left(\frac{5}{p}\right)=1,$ which is also a contradiction.
\end{itemize}

\vspace{0.5cm}
\noindent$\mathbfit{N}\mathbf{=33:}$
In \eqref{jed:fakt} we get
\begin{equation*} n^8d=Ds^2=m^8+10m^6n^2-8m^5n^3+47m^4n^4-40m^3n^5+82m^2n^6-44mn^7+33n^8.
\end{equation*}

Considering the possibilities of $m$ and $n$ modulo 8, we compute that $Ds^2\equiv 1\pmod{8},$ so from \Cref{lem:par2} we conclude that $D\equiv 1\pmod{8}$ and from \Cref{lem:par1} that the prime 2 splits.

We write $D=11^a\cdot p_1\cdot ... \cdot p_k, $ where $a\in\{0,1\}$ and $p_i\neq 2,$ since $D\equiv 1\pmod{8}$. If $a=1,$ then $D\equiv 0\pmod{11}.$ If $a=0,$ then $\left( \frac{D}{11} \right)=\left( \frac{p_1}{11} \right)\cdot ... \cdot\left( \frac{p_k}{11} \right),$ which is equal to 1 after using \Cref{thm2.1b} for $N=33$, which was proved independently. In this case we have that $\left( \frac{D}{11} \right)=1,$ therefore $11$ is not inert in $K.$	

A point of the form $(x_0, \sqrt{f_{33}(x_0)})$ with $x_0\in\Q$ is clearly defined over a real quadratic field, since $f_{33}(x_0)=x_0^8 + 10x_0^6 - 8x_0^5 + 47x_0^4 - 40x_0^3 + 82x_0^2 - 44x_0 + 33>0,$ for every $x_0.$ Therefore, $D>0.$	
\newline

\noindent$\mathbfit{N}\mathbf{=35:}$
In \eqref{jed:fakt} we get
\begin{equation*} n^8d=Ds^2=m^8-4m^7n-6m^6n^2-4m^5n^3-9m^4n^4+4m^3n^5-6m^2n^6+4mn^7+n^8.
\end{equation*}

Considering the possibilities of $m$ and $n$  modulo 4, we compute that $Ds^2\equiv 1 \pmod{4}$ and from \Cref{lem:par2} we conclude $D\equiv 1\pmod{4}.$ The fact that 2 is unramified now follows from \Cref{lem:par1}.

Looking at all the possibilities of $m$ and $n$ modulo 25, we compute that $Ds^2\equiv 1\pmod{5}$ or $Ds^2\equiv 5\pmod{25}$. It follows from \Cref{lem:nep2} that $D\equiv 0,1,4\pmod{5}$ and from \Cref{lem:nep1} that 5 is not inert.

\noindent Now want to eliminate the possibility $D\equiv 4\pmod{5}.$ If it were true, then for $s$ in $n^8d=Ds^2$ it holds $s^2\equiv 4\pmod 5$, so $s$ would be divisible by a prime $p$ such that $p\equiv 2,3 \pmod 5$, i.e. $\left(\frac{5}{p}\right)=-1.$

The expression $n^8d=Ds^2$ above factorizes as
\begin{equation*}
n^8d=Ds^2=\left(-m^2 - mn + n^2\right)\left(-m^6 + 5m^5n + 9m^3n^3 + 5mn^5 + n^6\right),
\end{equation*}
so $p$ has to divide one of the 2 factors on the right.
\begin{itemize}
	\item[$\bullet$] If $p$ divides the first factor, it also divides  $4(-m^2-nm+n^2) = (2m-n)^2 - 5m^2,$ then $\left(\frac{5}{p}\right)=1$, so $p\not\equiv 2,3 \pmod 5.$
	\item[$\bullet$] If $p$ divides the second factor, it also divides $4(-m^6+5m^5n+9m^3n^3+5mn^5+n^6) =
	(2n^3+5n^2m+5nm^2+4m^3)^2 - 5(3n^2m+nm^2+2m^3)^2,$ then $\left(\frac{5}{p}\right)=1$, so $p\not\equiv 2,3 \pmod 5.$
\end{itemize}

And in the end, considering the possibilities of $m$ and $n$ modulo 7, we compute that $Ds^2\equiv 1,2,4\pmod{7}.$ It follows from \Cref{lem:nep2} that $D\equiv 1,2,4\pmod{7}$ and from \Cref{lem:nep1} that 7 splits.
\newline

\noindent$\mathbfit{N}\mathbf{=39:}$
In \eqref{jed:fakt} we get
\begin{equation*} n^8d=Ds^2=m^8-6m^7n+3m^6n^2+12m^5n^3-23m^4n^4+12m^3n^5+3m^2n^6-6mn^7+n^8.
\end{equation*}

Considering the possibilities of $m$ and $n$  modulo 4, we compute that $Ds^2\equiv 1 \pmod{4}$ and from \Cref{lem:par2} we conclude $D\equiv 1\pmod{4}.$ The fact that 2 is unramified now follows from \Cref{lem:par1}.

We have that the right side of $n^8d=Ds^2$ above is congruent to $m^8-2m^4n^4+n^8=(m^4-n^4)^2$ modulo $3$.

\noindent Suppose first that $m\not \equiv n\pmod 3$. If $n\not \equiv 0 \pmod 3$ then $D$ is a square modulo $3$ and if $n\equiv 0 \pmod 3$ then it follows that $D\equiv 1 \pmod 3$ so $D$ is again a square modulo $3$.


\noindent Suppose now that $m\equiv n\pmod 3$. Then we run through all the possibilities of $m$ and $n$ modulo $81$ and compute that either $Ds^2$ is divisible by an odd power of $3$, so $D\equiv 0\pmod 3$, or $Ds^2\equiv 9k \pmod {81}$, where $k\not \equiv 0 \pmod {81}$ and $k$ is a square modulo $9$. Using \Cref{lem:nep2} we get that $D\equiv k \pmod {9}$, where $k$ is a square modulo $9$. Hence, in all cases we have $D\equiv 0,1 \pmod 3$ and from \Cref{lem:par1} we immediately see that 3 is not inert.

Considering the possibilities of $m$ and $n$  modulo 13, we compute that $Ds^2\equiv 1,3,4,9,10,12\pmod{13},$ and from \Cref{lem:nep2} we conclude $D\equiv 1,3,4,9,10,12\pmod{13}.$ The fact that 13 splits now follows from \Cref{lem:nep1}.
\newline

\noindent $\mathbfit{N}\mathbf{=40:}$
In \eqref{jed:fakt} we get
\begin{equation*} n^8d=Ds^2=m^8 + 8m^6n^2 -2m^4n^4 + 8m^2n^6 + n^8.
\end{equation*}

We write $n^8d=Ds^2$ as
$$ n^8d=Ds^2=(m^4-n^4)^2+8m^2n^2(m^4+n^4). $$
The integer $n$ has to be odd (otherwise $m$ and $n$ would both be even), and if $m$ is even, then we see that $Ds^2$ is an odd square modulo $8$. It follows from \Cref{lem:par2} that $D\equiv 1 \pmod 8$ and from \Cref{lem:par1} that 2 splits.

If $m$ and $n$ are both odd, then $Ds^2\equiv 16m^2n^2 \pmod{128}.$ From \Cref{lem:par2} we get that $D$ is an odd square modulo 8, i.e. $D\equiv 1 \pmod 8$. The fact that 2 splits now follows from \Cref{lem:par1}.

Considering the possibilities of $m$ and $n$ modulo 3, we compute that $Ds^2\equiv 1 \pmod{3}.$ Using \Cref{lem:nep2} we get $D\equiv1\pmod{3},$ and from \Cref{lem:nep1} we conclude that 3 splits.

Looking at all the possibilities of $m$ and $n$ modulo 5, we compute that $Ds^2\equiv 1,4\pmod{5}.$ Using \Cref{lem:nep2} we get $D\equiv 1,4\pmod{5},$ and from \Cref{lem:nep1} we conclude that 5 splits.

The proof of the fact that $D>0$ can be found in \cite[Theorem 4]{BN}.
\newline

\noindent $\mathbfit{N}\mathbf{=41:}$
In \eqref{jed:fakt} we get
\begin{equation*} n^8d=Ds^2=m^8-4m^7n-8m^6n^2+10m^5n^3+20m^4n^4+8m^3n^5-15m^2n^6-20mn^7-8n^8.
\end{equation*}

We write $D=41^a\cdot p_1\cdot ... \cdot p_k, $ where $a\in\{0,1\}.$ If $a=0,$ then $D\equiv 0\pmod{41}.$ If $a=1,$ then $\left( \frac{D}{41} \right)=\left( \frac{p_1}{41} \right)\cdot ... \cdot\left( \frac{p_k}{41} \right),$ which is equal to 1 after using \Cref{thm2.1b} for $N=41$, which was proved independently, and the fact that $\left( \frac{2}{41} \right)=1,$ in case one of the $p_i$ is 2. In this case we have that $\left( \frac{D}{41} \right)=1,$ and \Cref{lem:nep1} says that $41$ is not inert.
\newline

\noindent $\mathbfit{N}\mathbf{=46:}$
In \eqref{jed:fakt} we get
\begin{equation*}
\begin{split}
n^{12}d=Ds^2=&m^{12}-2m^{11}n+5m^{10}n^2+6m^9n^3-26m^8n^4+\\
& +84m^7n^5-113m^6n^6+134m^5n^7-64m^4n^8+26m^3n^9+12m^2n^{10}+8mn^{11}-7n^{12}.
\end{split}
\end{equation*}

Considering the possibilities of $m$ and $n$ modulo 512, we compute that $Ds^2\equiv 64\pmod{512}$ or $Ds^2\equiv 1\pmod{8}.$ Using \Cref{lem:par2}, in both cases we get $D\equiv 1\pmod{8},$ and from \Cref{lem:par1} we conclude that 2 splits.
\newline

\noindent $\mathbfit{N}\mathbf{=48:}$
In \eqref{jed:fakt} we get
\begin{equation*}
n^8d=Ds^2=m^8 + 14m^4n^4 + n^8.
\end{equation*}

We write $n^8d=Ds^2$ as
$$ n^8d=Ds^2=(m^4+n^4)^2+12m^4n^4.$$
If either $m$ or $n$ is even (forcing the other to be odd), then $Ds^2$ is an odd square modulo $8$. It follows from \Cref{lem:par2} that $D\equiv 1 \pmod 8$ and from \Cref{lem:par1} that 2 splits.

\noindent If $m$ and $n$ are both odd, then we compute  $Ds^2\equiv 16\pmod{128}.$ It follows from \Cref{lem:par2} that $D\equiv 1 \pmod 8$ and from \Cref{lem:par1} that 2 splits.

Considering the possibilities of $m$ and $n$ modulo 3, we compute that $Ds^2\equiv 1 \pmod{3}.$ Using \Cref{lem:nep2}, we get $D\equiv 1\pmod{3},$ and from \Cref{lem:nep1} we conclude that 3 splits.

Looking at all the possibilities of $m$ and $n$ modulo 5, we compute that $Ds^2\equiv 1\pmod{5}.$ Using \Cref{lem:nep2}, we get $D\equiv 1,4\pmod{5},$ and from \Cref{lem:nep1} we conclude that 5 splits.

A point of the form $(x_0, \sqrt{f_{48}(x_0)})$ with $x_0\in\Q$ is clearly defined over a real quadratic field, since $f_{48}(x_0)=x_0^8+14x_0^4+1>0,$ for every $x_0.$ Therefore, $D>0.$
\newline

\noindent $\mathbfit{N}\mathbf{=50:}$
In \eqref{jed:fakt} we get
\begin{equation*}
n^6d=Ds^2=m^6 -4m^5n -10m^3n^3 -4mn^5 + n^6.
\end{equation*}

Considering the possibilities of $m$ and $n$ modulo 5, we compute that $Ds^2\equiv 0,1,4 \pmod{5}.$ Using \Cref{lem:nep2}, we get $D\equiv 0,1,4\pmod{5},$ and from \Cref{lem:nep1} we conclude that 5 is not inert.

We have
$$n^6d=Ds^2\equiv (m^3-n^3)^2\pmod 4.$$
If either $m$ or $n$ is even it follows that $D$ is odd. If $m$ and $n$ are both odd, we compute that $Ds^2\equiv 4\pmod{16}, Ds^2\equiv 16\pmod{32}$ or $Ds^2\equiv 64\pmod{128}.$ Using \Cref{lem:par2}, in all cases we get that $D$ is odd.
\end{proof}

\begin{remark}\label{rem:ozman}
	
	We now mention two papers \cite{gonz,EO} (note that \cite{EO} builds on previous work of Clark \cite{clark}) that have some overlap with our results from \Cref{thm2.1a} and show which of our results can be proved using their methods.
	
	Non-exceptional points on curves $X_0(N)$ are of the form $(x, y\sqrt{d}),$ where $x,y\in\Q$. This gives us the point $(x,y)$ on the quadratic twist $X^d_0(N)(\Q)$ and hence $X^d_0(N)(\Q_p)\neq \emptyset.$ Now some of the entries in Table \ref{table:ram} can be alternatively proved using the results of Ozman \cite[Theorem 1.1]{EO}, and those are
	\begin{itemize}
		\item[$\bullet$] for $N=26$: $13$ is not inert,
		\item[$\bullet$] for $N=29$: $29$ is not inert,
		\item[$\bullet$] for $N=30$: $2$, $3$ split and 5 is not inert,
		\item[$\bullet$] for $N=35$: $7$ splits, $2$ is unramified and 5 is not inert,
		\item[$\bullet$] for $N=39$: $13$ splits and $2$ is unramified,
		\item[$\bullet$] for $N=41$: $41$ is not inert,
		\item[$\bullet$] for $N=46$: $2$ is unramified.
	\end{itemize}

	A $\Q$-curve is an elliptic curve defined over a number field that is geometrically isogenous to each of its Galois conjugates. The \textit{degree} of a $\Q$-curve over a quadratic field is the degree of a cyclic isogeny to its Galois conjugate. Gonz\'alez proves the following statement \cite[Proposition 1.1]{gonz}:
	
	\textit{Assume that there exists a quadratic $\Q$-curve of degree $d$ defined over some quadratic field $K$. Then every divisor $N_1\mid d$ such that
		$$N_1 \equiv 1 \pmod 4 \quad \text{or} \quad N_1 \text { is even and } d/N_1 \equiv 3 \pmod 4$$
		is a norm of the field $K$.}

	For our values of $N$ all but finitely many known exceptions of elliptic curves with $N$-isogenies over quadratic fields are $\Q$-curves (as proved by Bruin and Najman \cite{BN}). Note that we do not use the fact that the curves we consider are $\Q$-curves in any essential way; we only use the fact that almost all the quadratic points on the modular curves $X_0(N):y^2=f_N(x)$ are of the form $(x_0,\sqrt{f_N(x_0)})$ for $x_0\in \Q$ (and from this fact Bruin and Najman proved that the corresponding elliptic curves are $\Q$-curves).
	
	After noting that a non-exceptional quadratic point on $X_0(N)$ corresponds to a $\Q$-curve of degree $d$, where $d$ can be obtained from the tables in \cite{BN}), and applying Gonz\'alez' proposition, we obtain $p$ is not inert in a quadratic field $K:=\Q(\sqrt{D})$ generated by a non-exceptional point on $X_0(N)$ for the following pairs $(N,p)$:
	$$(N,p) \in \left\{(26,13), (29,29), (30,5), (35,5), (41,41), (50,5)  \right\}.$$
	In all of the pairs above we have $d=N$ except for $N=30$, where $d=15$. \qed
\end{remark}

\begin{theorem}\label{thm2.1b}
	Let $K$ be a quadratic field over which $X_0(N)$ has a non-exceptional non-cuspidal point. For the pairs of $N$ and $a$ indicated in the table, if a prime $p$ ramifies in $K$, then $a$ is a square modulo $p$.
	
		\begin{table}[h]
		\begin{tabular}{cccccccccccc}\hlinewd{0.8pt}\\[-0.8em]
			\multicolumn{1}{c}{$\mathbfit{N} $} & \multicolumn{1}{c}{$26$}&\multicolumn{1}{c}{$28$}&\multicolumn{1}{c}{$29$}&\multicolumn{1}{c}{$30$}&\multicolumn{1}{c}{$33$}&\multicolumn{1}{c}{$35$}&\multicolumn{1}{c}{$39$}&\multicolumn{1}{c}{$40$}&\multicolumn{1}{c}{$41$}&\multicolumn{1}{c}{$48$}&\multicolumn{1}{c}{$50$} \\ \hdashline\\[-0.8em]
			$\:\:\:\:\:\:\:\:a\:\:\:\:\:\:\:\:$                       & $13$  & $-7^*$& $29$ & $5$ & $-11$& $5$ & $13$ & $-1,5$& $41$ & $-1,3$& $5$  \\
			\hlinewd{0.8pt}
			\multicolumn{12}{l}{\hspace{0.6cm}$^{*}$ \textup{-the statement of the theorem is true with the exception of $p=2$}}   \\
			\hlinewd{0.8pt}
		\end{tabular}
		\vspace{0.1cm}
		\caption{}
		\label{}
	\end{table}
\end{theorem}

We first prove two lemmas that will be useful in the proof of \Cref{thm2.1b}.

\begin{lemma} \label{lem:referee}
	Suppose $f_N$ factorizes as $f_N=\prod_{i \in I}f_{N,i}$, where $f_{N,i}\in \Z[x]$ are irreducible factors of degree $2$ or $3$ and $p\nmid a_{N,0}$. If $p$ ramifies in $K$, then there exists an $i\in I$ such that $\Delta(f_{N,i})$ is a square modulo $p$.
\end{lemma}
\begin{proof}
	Assume that $p$ ramifies in $K$; then by \Cref{lem:nep1} it follows that $p|D$. If $p|n$, then it would follow that $p|m$, which is a contradiction, so we conclude that $p\nmid n$. Dividing out \eqref{jed:fakt} by $n$, we see that $m/n$ is a root of $f_N$ modulo $p$ and hence there exists an $i\in I$ such that $m/n$ is a root of $f_{N,i}$ modulo $p$.
	
	If $f_{N,i}$ is of degree 2 or 3, the formulas for the roots of quadratic and cubic polynomials imply that $\sqrt{\Delta(f_{N,i})}$ is defined over $\F_p$, which proves the statement.
\end{proof}

\begin{remark}
	Note that the statement of \Cref{thm2.1b} can be proved with the previous lemma only for $(N,a)=(28,-7)$. We have $f_{28}(x)=(2x^2-3x+2)(x^2-x+2)(2x^2-x+1)$ and $\Delta(f_{28,i})=-7,$ for each $i$.
\end{remark}

As mentioned in the remark, \Cref{lem:referee} is not enough to prove all of the statements in \Cref{thm2.1b}, so we provide a generalization.

\begin{lemma}\label{lem:referee2}
	Let $f_N=\prod_{i \in I}f_{N,i}$ be the decomposition into irreducible factors, with $f_{N,i}\in \Z[x]$. Assume that there exists a quadratic field $K_0$ such that each $f_{N,i}$ becomes reducible in $K_0[x]$ and let $p$ be an odd prime such that $(p, \Delta(f_{N,i}))=1$ for all $i$. Then if $p$ ramifies in $K$ it follows that $\Delta(K_0)$ is a square modulo $p$, i.e. $p$ is not inert in $K_0$.
\end{lemma}
\begin{proof}
	Let $\sigma$ be the generator of $\Gal(K_0/\Q)$ and $f_{N,i,K_0}\in K_0[x]$ an irreducible factor of $f_{N,i}$. Then we obviously have \begin{equation}\label{jed:pol}
	f_{N,i, K_0}(f_{N,i,K_0})^\sigma=f_{N,i}.
	\end{equation}
	
	Assume that $p$ ramifies in $K$. We will prove the lemma by contradiction, so we assume that $p$ is inert in $K_0$. As in the proof of \Cref{lem:referee} we conclude that $f_{N,i}$ has a root $a$ in $\F_p$ for some $i$. Hence $a$ is a root of one of the factors on the left in \eqref{jed:pol}. Assume without loss of generality that $a$ is a root of $f_{N,i}$ in $\F_p$.
	
	Let $\fp$ be the prime of $K$ above $p$ and denote by $\F_\fp:=\OO_{K_0}/\fp$ the residue field of $\fp$. Let $\tau=\Gal(\F_\fp/\F_p)$ and denote by $\overline f$ the reduction of a polynomial $f\in K_0[x]$ modulo $\fp$; then we have $\overline{f^\sigma}=\overline{f}^\tau$. Hence $a^\tau$ is a root of $\overline f_{N,i}^\tau$. But since $a\in \F_p$, it follows that $a=a^\tau$ and hence from \eqref{jed:pol} it follows that $a$ is a double root of $f_{N,i}$ over $\F_p$ and hence $\Delta(f_{N,i})$ is divisible by $p$, which is in contradiction with the assumption $(p, \Delta(f_{N,i}))=1.$
\end{proof}

\begin{proof}[Proof of \Cref{thm2.1b}] Let $f_N=\Pi_{i}f_{N,i}$ be the factorization of $f_N$ in $\Z\left[X\right],$ as in \Cref{tab:disc}. Table \ref{tab:tab2}, which can be computed with the accompanying \href{https://web.math.pmf.unizg.hr/~atrbovi/magma/magma2/Table3}{Magma code}, contains for each $N$ the number $a$ such that every $f_{N,i}$ becomes reducible in $\Q(\sqrt{a}),$ the factorization in $\Q(\sqrt{a})$ and discriminants of each $f_{N,i}.$ Using Lemma \ref{lem:referee2} we immediately get that if an odd prime $p$ such that $(p, \Delta(f_{N,i}))=1$ ramifies in $K$, then $a$ is a square modulo $p$. For $p=2$ and $p$ that are not coprime to every $\Delta(f_{N,i})$ and can ramify (this can be checked in \Cref{thm:unr}, which is proved independently) we can explicitly verify that $\left(\frac{a}{p}\right)\neq -1.$
\end{proof}

\begin{theorem}\label{thm2.1c}
	Let $K$ be a quadratic field over which $X_0(N)$ has a non-exceptional non-cuspidal point. For the pairs of $N$ and $a$ indicated in the table, if $p\neq 2$ is a prime such that $a$ is a square modulo $p$, then there exist infinitely many quadratic fields generated by a point on $X_0(N)$ in which $p$ ramifies.

		\begin{table}[h]
		\begin{tabular}{ccccc}\hlinewd{0.8pt}\\[-0.8em]
			\multicolumn{1}{c}{$\mathbfit{N} $} &\multicolumn{1}{c}{$28$}&\multicolumn{1}{c}{$30$}&\multicolumn{1}{c}{$33$}&\multicolumn{1}{c}{$35$} \\ \hdashline\\[-0.8em]
			$\:\:\:\:\:\:\:\:$a$\:\:\:\:\:\:\:\:$                       & $-7$ & $5$ & $-11$& $5$  \\
			\hlinewd{0.8pt}
		\end{tabular}
		\vspace{0.1cm}
		\caption{}
		\label{}
	\end{table}
\end{theorem}	
	
\begin{proof}
	For all pairs of $N$ and $a$, in Table \ref{tab:tab2} we have the factorizations of $f_N$ where some of the factors are linear over $\Q(\sqrt{a}).$ Therefore, $f_N$ has a root over each $\F_p$ such that $\sqrt{a}$ is defined modulo $p,$ i.e. such that $a$ is a square modulo $p$.
	
	\noindent If $x_0\in \Z$ is a root of $f_N$ such that $f_N(x_0)\equiv 0\pmod p,$ then $f_N(x_0+kp)\equiv 0\pmod p ,\:\: k=0,...,p-1.$ If $p>\deg f_N,$ we have $f_N(x_0+kp)\not\equiv 0\pmod {p^2}$ for at least one value of $k$.
	Now we know that for $p>\deg f_N$ there exists $a \in \Z$ be such that $f_N(a)\equiv 0 \pmod p$ and $f_N(a)\not\equiv 0 \pmod {p^2}.$ For smaller values of $p$, with exception of $p=2,$ one can explicitly check that this claim remains true. Therefore, $p$ ramifies in $\mathbb{Q}(\sqrt{f_N(a)}).$

	\noindent It remains to show that there are infinitely many quadratic fields such that $p$ ramifies. Let $S=\{u\in \Z: u\equiv a \pmod {p^2}\}$. Obviously $f_N(u)\equiv 0\pmod p$ and $f_N(u)\not\equiv 0 \pmod {p^2}$ for all $u \in S$. Let $d_u$ be the squarefree part of $f_N(u)$; the quadratic point $(u,\sqrt{f_N(u)})$ will be defined over $\Q(\sqrt{d_u})$. After writing $f_N(u)=d_us_u^2$ for some $s_u\in \Z$, we observe that $(u,s_u)$ is a rational point on the quadratic twists $C_N^{d_u}$ of $X_0(N)$,
	$$C_N^{d_u}:d_uy^2=f_N(x).$$
	Since each $C_N^{d_u}$ is of genus $\geq 2$, by Faltings' theorem it follows that $C_N^{d_u}(\Q)$ is finite and hence $\{d_u:u \in S\}$ is infinite, proving the claim.
\end{proof}

	\begin{table}[h]
		\begin{tabular}{ccc}\hlinewd{0.8pt}\\[-0.8em]
			\multicolumn{1}{c}{$\mathbfit{N} $}&\multicolumn{1}{c}{\textup{\textbf{factorization $\mathbfit{f_N=\prod_{i \in I}f_{N,i}}$ in $\Z[x]$}}}&\multicolumn{1}{c}{$\Delta(\mathbfit{f_{N,i}}) $} \\ \hlinewd{0.8pt}\\[-0.8em]
			$\:\:\:\mathbf{26}\:\:\:$  & $x^6 - 8x^5 + 8x^4 - 18x^3 + 8x^2 - 8x + 1$& $2^{20}\cdot 13^3$ \\ \hdashline\\[-0.8em]
			$\mathbf{28}$                     &  $(2x^2 - 3x + 2)\times$ & $-7$   \\
			&  $\times (x^2 - x + 2)\times$ & $-7$   \\
			&  $\times (2x^2 - x + 1)$ & $-7$   \\ \hdashline\\[-0.8em]
			$\mathbf{29}$                    & $x^6 - 4x^5 - 12x^4 + 2x^3 + 8x^2 + 8x - 7$&$2^{12}\cdot 29^5$      \\ \hdashline\\[-0.8em]
			$\mathbf{30}$                     &  $(x^2 + 3x + 1)\times$ & $5$   \\
			&  $\times (x^2 + 6x + 4)\times$ & $2^2\cdot 5$   \\
			&  $\times (x^4 + 5x^3 + 11x^2 + 10x + 4)$ & $2^2\cdot 3^2 \cdot 5^2$   \\ \hdashline\\[-0.8em]
			$\mathbf{33}$
			&  $(x^2 - x + 3)\times$ & $-11$   \\
			&  $\times (x^6 + x^5 + 8x^4 - 3x^3 + 20x^2 - 11x + 11)$ & $-2^8\cdot 3^6 \cdot 11^5$   \\ \hdashline\\[-0.8em]
			$\mathbf{35}$
			&  $(x^2 + x - 1)\times$ & $5$   \\
			&  $\times (x^6 - 5x^5 - 9x^3 - 5x - 1)$ & $2^8\cdot 5^7 \cdot 7^2$   \\ \hdashline\\[-0.8em]
			$\mathbf{39}$
			&  $(x^4 - 7x^3 + 11x^2 - 7x + 1)\times$ & $-3^3 \cdot 13^2$   \\
			&  $\times (x^4 + x^3 - x^2 + x + 1)$ & $-3 \cdot 13^2$   \\ \hdashline\\[-0.8em]
			$\mathbf{40}$                     &  $x^8 + 8x^6 - 2x^4 + 8x^2 + 1$ & $2^{40}\cdot 5^4$   \\ \hdashline\\[-0.8em]
			$\mathbf{41}$                     & $x^8 - 4x^7 - 8x^6 + 10x^5 + 20x^4 + 8x^3 - 15x^2 - 20x - 8$ & $-2^{16}\cdot 41^6$ \\ \hdashline\\[-0.8em]
			$\mathbf{48}$
			&  $(x^4 - 2x^3 + 2x^2 + 2x + 1)\times$ & $2^8 \cdot 3^2$   \\
			&  $\times (x^4 + 2x^3 + 2x^2 - 2x + 1)$ & $2^8 \cdot 3^2$   \\ \hdashline\\[-0.8em]
			$\mathbf{50}$                  &  $x^6 - 4x^5 - 10x^3 - 4x + 1$ & $2^{16}\cdot 5^5$    \\ \hlinewd{0.8pt}
		\end{tabular}
		\vspace{0.1cm}
		\caption{Factorizations $f_N=\prod_{i \in I}f_{N,i}$
		in $\Z[x]$ and the discriminants of $\mathit{f_{N,i}}$ from the statement of \Cref{lem:referee2}.}
		\label{tab:disc}
	\end{table}

	\newpage
	\begin{table}[h]
		\begin{tabular}{crc}\hlinewd{0.8pt}\\[-0.8em]
			\multicolumn{1}{c}{$\mathbfit{N} $} & \multicolumn{1}{r}{\textup{\textbf{$a$}}}&\multicolumn{1}{c}{\textup{\textbf{factorization of $\mathbfit{f_N}$ in $\Q(\sqrt{a})$}}} \\ \hlinewd{0.8pt}\\[-0.8em]
			$\:\:\:\mathbf{26}\:\:\:$  &$13$& \thead{$\left((x^3+(-\sqrt{13} - 4)x^2 + \frac{1}{2}(\sqrt{13} + 5)x + \frac{1}{2}(-3\sqrt{13} - 11)\right) \times $ \\ $\times \left(x^3+(\sqrt{13} - 4)x^2 + \frac{1}{2}(-\sqrt{13} + 5)x + \frac{1}{2}(3\sqrt{13} - 11)\right)$ } \\ \hdashline\\[-0.8em]
			$\mathbf{28}$                       &$-7$&  \thead{ $\left(x + \frac{1}{2}(-\sqrt{-7} - 1)\right)\left(x + \frac{1}{4}(-\sqrt{-7} - 3\right)\times$ \\ $\times\left(x + \frac{1}{4}(-\sqrt{-7} - 1)\right)\left(x + \frac{1}{4}(\sqrt{-7} - 3)\right)\times$ \\ $\times\left(x + \frac{1}{4}(-\sqrt{-7} - 1)\right)\left(x + \frac{1}{2}(\sqrt{-7} - 1)\right)$ }   \\ \hdashline\\[-0.8em]
			$\mathbf{29}$                     &$29$& \thead{$\left(x^3 + (-\sqrt{29}-2)x^2 + \frac{1}{2}(\sqrt{29} + 13)x + \frac{1}{2}(-\sqrt{29} - 1)\right) \times $ \\ $\times \left(x^3 + (\sqrt{29}-2)x^2 + \frac{1}{2}(-\sqrt{29} + 13)x + \frac{1}{2}(\sqrt{29} - 1)\right)$ }     \\ \hdashline\\[-0.8em]
			$\mathbf{30}$                     &$5$& \thead{$\left(x - \sqrt{5} + 3\right)\left(  x + \frac{1}{2}(-\sqrt{5} + 3)\right)\times $\\$\times \left(  x + \frac{1}{2}(\sqrt{5} + 3)\right)\left( x + \sqrt{5} + 3 \right)\times $ \\ $\times \left(  x^2 + \frac{1}{2}(-\sqrt{5} + 5)x - \sqrt{5} + 3\right)\left(  x^2 + \frac{1}{2}(\sqrt{5} + 5)x + \sqrt{5} + 3\right)$ }     \\ \hdashline\\[-0.8em]
			$\mathbf{33}$                     &$-11$& \thead{$\left(x + \frac{1}{2}(-\sqrt{-11} - 1)\right)\left(x + \frac{1}{2}(\sqrt{-11} - 1)\right)\times $\\$\times \left(x^3 + \frac{1}{2}(-\sqrt{-11} + 1)x^2 + \frac{1}{2}(\sqrt{-11} + 5)x - \sqrt{-11} \right)\times $ \\ $\times \left(x^3 + \frac{1}{2}(\sqrt{-11} + 1)x^2 + \frac{1}{2}(-\sqrt{-11} + 5)x + \sqrt{-11} \right)$ } \\ \hdashline\\[-0.8em]
			$\mathbf{35}$                     &$5$& \thead{$\left(x + \frac{1}{2}(-\sqrt{5} + 1)\right)\left(x + \frac{1}{2}(\sqrt{5} + 1)\right)\times $\\$\times \left(x^3 + \frac{1}{2}(-3\sqrt{5} - 5)x^2 + \frac{1}{2}(\sqrt{5} + 5)x - \sqrt{5} - 2 \right)\times $ \\ $\times \left(x^3 + \frac{1}{2}(3\sqrt{5} - 5)x^2 + \frac{1}{2}(-\sqrt{5} + 5)x + \sqrt{5} - 2 \right)$ } \\ \hdashline\\[-0.8em]
			$\mathbf{39}$                     &$13$& \thead{$\left( x^2 + \frac{1}{2}(-\sqrt{13} - 7)x + 1 \right) \left(x^2 + \frac{1}{2}(-\sqrt{13} +1)x + 1\right) \times $ \\ $\times \left( x^2 + \frac{1}{2}(\sqrt{13} - 7)x + 1 \right) \left(x^2 + \frac{1}{2}(\sqrt{13} +1)x + 1\right)$ }    \\ \hdashline\\[-0.8em]
			$\mathbf{40}$                     &\thead{$\:\:\:\:\:-1$ \\ $\:$\\ $\:$ \\ $\:\:\:\:\:\:\:5$}&  \thead{$\left(x^4 - 2\sqrt{-1}x^3 + 2x^2 + 2\sqrt{-1}x + 1 \right) \times $ \\ $\times \left(x^4 + 2\sqrt{-1}x^3 + 2x^2 - 2\sqrt{-1}x + 1 \right)$\\ $\:$ \\ $\left(x^4 + (-2\sqrt{5} + 4)x^2 + 1 \right)  \left(x^4 + (2\sqrt{5} + 4)x^2 + 1 \right)$ }  \\ \hdashline\\[-0.8em]
			$\mathbf{41}$                     &$41$&  \thead{$\left(x^4 - 2x^3 + (-\sqrt{41} - 6)x^2 + (-\sqrt{41} - 7)x + \frac{1}{2}(-\sqrt{41} - 3)\right) \times $ \\ $\times \left(x^4 - 2x^3 + (\sqrt{41} - 6)x^2 + (\sqrt{41} - 7)x + \frac{1}{2}(\sqrt{41} - 3)\right)$ } \\ \hdashline\\[-0.8em]
			$\mathbf{48}$                     &\thead{$\:\:\:\:\:-1$ \\ $\:$\\ $\:$ \\ $\:\:\:\:\:\:\:3$}&	\thead{$\left( x^2 + (-\sqrt{-1} - 1)x - \sqrt{-1} \right) \left(x^2 + (-\sqrt{-1} + 1)x + \sqrt{-1}\right) \times $ \\ $\times \left( x^2 + (\sqrt{-1} - 1)x + \sqrt{-1} \right) \left(x^2 + (\sqrt{-1} + 1)x - \sqrt{-1}\right)$\\ $\:$ \\ $\left( x^2 + (-\sqrt{3} - 1)x + \sqrt{3} + 2 \right) \left(x^2 + (-\sqrt{3} + 1)x - \sqrt{3} + 2\right) \times $ \\ $\times \left( x^2 + (+\sqrt{3} - 1)x - \sqrt{3} + 2 \right) \left(x^2 + (\sqrt{3} + 1)x + \sqrt{3} + 2\right)$ }   \\ \hdashline\\[-0.8em]
			$\mathbf{50}$                     &$5$&  \thead{$\left(x^3 + (-\sqrt{5} - 2)x^2 + \frac{1}{2}(-\sqrt{5} + 1)x + \frac{1}{2}(-\sqrt{5} - 3)\right) \times $ \\ $\times \left(x^3 + (\sqrt{5} - 2)x^2 + \frac{1}{2}(\sqrt{5} + 1)x + \frac{1}{2}(\sqrt{5} - 3)\right)$ }  \\ \hlinewd{0.8pt}
		\end{tabular}
		\vspace{0.1cm}
		\caption{Factorizations of $\mathit{f_N}$ in $\Q(\sqrt{a})$ used in the proof of  \Cref{lem:referee2} and \Cref{thm2.1c}.}
		\label{tab:tab2}
	\end{table}

\newpage

\begin{theorem}\label{thm:unr}
In \Cref{table4} below, we list the primes $p\leq 100$ which are unramified for all quadratic fields generated by quadratic points $X_0(N)$, for $N\in\{22,23,26,29,30,31,33,35,39,40,41,46,47,48,50,59,71\}.$
\begin{table}[h]

	\begin{tabular}{cl}\hlinewd{0.8pt}	\\[-0.8em]

		\multicolumn{1}{c}{$\mathbfit{N} $} & \multicolumn{1}{l}{\textup{\textbf{unramified primes}}} \\ \hlinewd{0.8pt}
		\\[-0.8em]
		$\:\:\:\:\mathbf{22}\:\:\:\:$                        & $3, 5, 23, 31, 37, 59, 67, 71, 89, 97$          \\ \hdashline
		\\[-0.8em]
		$\mathbf{23}$                        & $2, 3, 13, 29, 31, 41, 47, 71, 73$								                                              \\ \hdashline
		\\[-0.8em]
		$\mathbf{26}$                        & $3, 5, 7, 11, 17, 19, 31, 37, 41, 43, 47, 59, 67, 71, 73, 83, 89, 97								$                                              \\ \hdashline
		\\[-0.8em]
		$\mathbf{28}$                        & $3, 5, 13, 17, 19, 31, 41, 47, 59, 61, 73, 83, 89, 97								$                                              \\ \hdashline
		\\[-0.8em]
		$\mathbf{29}$                        & $3, 5, 11, 13, 17, 19, 31, 37, 41, 43, 47, 53, 61, 73, 79, 89, 97								$                                              \\ \hdashline
		\\[-0.8em]
		$\mathbf{30}$                        & $2, 3, 7, 13, 17, 23, 37, 43, 47, 53, 67, 73, 83, 97								$                                              \\ \hdashline
		\\[-0.8em]
		$\mathbf{31}$                        & $2, 5, 7, 19, 41, 59, 71, 97								$                                              \\ \hdashline
		\\[-0.8em]
		$\mathbf{33}$                        & $2, 7, 13, 17, 19, 29, 41, 43, 61, 73, 79, 83								$                                              \\ \hdashline
		\\[-0.8em]
		$\mathbf{35}$                        & $2, 3, 7, 13, 17, 23, 37, 43, 47, 53, 67, 73, 83, 97								$                                              \\ \hdashline
		\\[-0.8em]
		$\mathbf{39}$                        & $2, 5, 7, 11, 13, 19, 31, 37, 41, 47, 59, 61, 67, 71, 73, 79, 83, 89, 97								$                                              \\ \hdashline
		\\[-0.8em]
		$\mathbf{40}$                        & $2, 3, 5, 7, 11, 13, 17, 19, 23, 31, 37, 41, 43, 47, 53, 59, 61, 67, 71, 73,
		79, 83, 97$						                                              \\ \hdashline
		\\[-0.8em]
		$\mathbf{41}$                        & $3, 5, 7, 11, 13, 17, 19, 29, 37, 47, 53, 61, 67, 71, 73, 79, 89, 97								$                                              \\ \hdashline
		\\[-0.8em]
		$\mathbf{46}$                        & $2, 3, 13, 29, 31, 41, 47, 71, 73								$                                              \\ \hdashline
		\\[-0.8em]
		$\mathbf{47}$                        & $2, 3, 7, 17, 37, 53, 59, 61, 71, 79, 89, 97								$                                              \\ \hdashline
		\\[-0.8em]
		$\mathbf{48}$                        & $2, 3, 5, 7, 11, 17, 19, 23, 29, 31, 41, 43, 47, 53, 59, 67, 71, 79, 83, 89							$                                              \\ \hdashline
		\\[-0.8em]
		$\mathbf{50}$                        & $3, 7, 11, 13, 17, 19, 23, 37, 41, 43, 47, 53, 67, 73, 83, 89, 97								$                                              \\ \hdashline
		\\[-0.8em]
		$\mathbf{59}$                        & $3, 5, 7, 19, 29, 41, 53, 79								$                                              \\ \hdashline
		\\[-0.8em]
		$\mathbf{71}$                        & $2, 3, 5, 19, 29, 37, 43, 73, 79, 83, 89$\\ \hlinewd{0.8pt}
	\end{tabular}
\vspace{0.1cm}
\caption{Primes up to 100 that do not ramify in quadratic fields over which $X_0(N)$ has a point.}
\label{table4}
\end{table}
\end{theorem}
\begin{proof}
The proofs of all the facts listed are easy and all basically the same; take some prime $p$ in the table above. Using the notation as in \ref{jed:fakt}, we compute that $f_N$ has no linear factor modulo $p$. It follows that  $n^{2k}d\not\equiv 0 \pmod p$ for any positive integer $k$, which gives us that $D\not\equiv 0 \pmod p$ and hence $p$ is unramified.
\end{proof}

\section{Splitting of $2$ in cubic fields generated by cubic points of $X_1(2,14)$}
\label{sec:cubic}

Let us fix the following notation for the remainder of this section. Denote $X:=X_1(2,14)$ and $Y:=Y_1(2,14)$. Let $\phi:X_1(2,14)\rightarrow X_1(14)$ be the forgetful map sending $(E,P,Q,R)\in X$ with $P$ and $Q$ of order 2 and $R$ of order 7 to $(E,P,R)\in X_1(14)$. Let $K$ be a cubic number field over which $X$ has a non-cuspidal point $x=(E,P,Q,R)$ and let $\PP$ be a prime above $p$. By \cite[Theorem 1.2]{BN2}, $K$ is a cyclic cubic field and $E$ is a base change of an elliptic curve defined over $\Q$. Denote by $\overline x$ the reduction of $x$ mod $\PP$.


The main result of this section is the following proposition.

\begin{proposition}\label{thm:main_sec3}
	Let $K$ be a cyclic cubic field and $E/\Q$ an elliptic curve such that $E(K)_{tors}\simeq \Z/2\Z \times \Z/14 \Z$. Let $q=2$ or $q \equiv \pm 1 \pmod 7$ be a rational prime such that $E$ has multiplicative reduction in $q$. Then $q$ splits in $K$.
\end{proposition}

Note that, as has been mentioned, the assumptions that fact that $K$ is cyclic and that $E$ is a base change od an elliptic curve defined over $\Q$ are not necessary as they both follow from \cite[Theorem 1.2]{BN2}, but have been included for the proof to be easier to follow.

We prove \Cref{thm:main_sec3} by first showing that the reduction $\overline x$ of $x$ modulo a prime $\PP$ above $q$ is defined over $\F_q$. On the other hand, by considering the morphisms from $X$ to $X_1(14)$ and $\P^1$, we show that if $q$ was ramified or inert in $K$, then the Frobenius of $\F_\PP$ would have to act nontrivially on $\overline{x}$, which leads to a contradiction.

The curve $X$ has the following affine model \cite[Proposition 3.7]{BN2}:
\begin{equation}\label{X214}
  X: f(u,v)=(u^3 + u^2 - 2u - 1)v(v + 1)
+ (v^3 + v^2 - 2v - 1)u(u + 1) = 0.
\end{equation}
The curve $X$ has 18 cusps, 9 which are defined over $\Q$ and $9$ which are contained in $3$ Galois orbits, each consisting of 3 points defined over $\Q(\zeta_7)^+$. 

\begin{remark}
As has been mentioned, $E$ is a base change of an elliptic curve over $\Q$, so in \Cref{thm:main_sec3} and the remainder of the section, when we consider the reduction of $E$ (or more precisely its global minimal model) modulo a rational prime $p$, we will consider $E$ to be defined over $\Q$. The same comment applies to the elliptic curve $X_1(14)$. We will also consider the reduction of (the global minimal model over $\Q$ of) $X_1(14)$ modulo a prime $\PP$ of $K$. This is also well defined since $X_1(14)$ is semistable; in  particular $X_1(14)$ has multiplicative reduction at all $\PP$ dividing $2$ and $7$ and good reduction at all other primes.

When considering the reduction of $X$ modulo a prime $\PP$, we will always mean the reduction of the model given in \eqref{X214}, viewed as a subscheme of $\P^1 \times  \P^1$ over $\Z$.

\end{remark}

Let $\tau$ and $\omega$ be automorphisms of $X$, where the moduli interpretation of $\tau$ is that it acts as a permutation of order $3$ on the points of order $2$ of $E$ and trivially on the point of order 7, and where the moduli interpretation of $\omega$ is that it acts trivially on the points of order 2 and as multiplication by 2 on the point of order $7$. Let $\alpha:=\omega\tau$ and $\beta:=\omega \tau^2$.

From \cite[Section 3]{BN2} it follows that the only maps of degree $3$ from $X$ to $\P^1$ are quotienting out by subgroups generated by $\alpha$ and $\beta$ (an automorphism of $X$ interchanges these two maps) and that all non-cuspidal cubic points on $X$ are inverse images of $\P^1(\Q)$ with respect to these maps. Also, both $\alpha$ and $\beta$ act without fixed points on the cusps.



Thus, one can write every elliptic curve with $\Z/2\Z \times \Z/14\Z$ torsion over a cubic field as $E_u$ for some $u\in \Q$. We do not display the model for $E_u$ as it contains huge coefficients, but it can be found in the accompanying \href{https://web.math.pmf.unizg.hr/~atrbovi/magma/magma2/X1(2,14)}{Magma code}. In \cite{BN2} it is proved that the curve $E:=E_u$ is a base change of an elliptic curve defined over $\Q$.

We have
$$j(u)=\frac{(u^2 + u + 1)^3(u^6 + u^5 + 2u^4 + 9u^3 + 12u^2 + 5u + 1)f_{12}(u)^3}{u^{14}(u+1)^{14}(u^3 + u^2 - 2u - 1)^2},$$
where
$$f_{12}(u)=u^{12} + 4u^{11} + 3u^{10} - 4u^9 + 6u^7 - 17u^6 - 30u^5 + 6u^4 + 34u^3 +
25u^2 + 8u + 1,$$ and
\begin{align*}
\Delta(u)&=\frac{u^{14}(u+1)^{14}(u^3 + u^2 - 2u - 1)^2}{h_{12}(u)^{12}},\\
c_4(u)&=\frac{g_2(u)g_6(u)g_{12}(u)}{h_{12}(u)^{4}},
\end{align*}
where
\begin{align*}
  h_{12}(u)&=u^{12} + 6u^{11} + 13u^{10} + 14u^9 + 20u^8 + 44u^7 + 53u^6 + 36u^5 +
        34u^4 + 36u^3 + 19u^2 + 6u + 1, \\
  g_2(u) &= u^2 + u + 1, \quad g_6(u)=u^6 + u^5 + 2u^4 + 9u^3 + 12u^2 + 5u + 1,\\
 g_{12}(u)&= u^{12} + 4u^{11} + 3u^{10} - 4u^9 + 6u^7 - 17u^6 - 30u^5 + 6u^4 + 34u^3 +
        25u^2 + 8u + 1.
\end{align*}

Before proceeding to prove \Cref{thm:main_sec3}, we prove two propositions about the reduction types of $E$ at primes of multiplicative reduction.

\begin{proposition}\label{j-inv}
	Suppose $E$ has multiplicative reduction at a rational prime $p$. Then either the reduction is of type $I_{14k}$ for some $k$, or $p\equiv \pm 1 \pmod 7,$ in which case the reduction is $I_{2k}.$
\end{proposition}

\begin{proof}
A necessary condition for $E$ to have multiplicative reduction at $p$ is for $v_p(\Delta(u))>0$, so we consider only $u\in \Q$ which satisfy this. Let $\res(f,g)$ denote the resultant of the polynomials $f$ and $g$.  If $v_p(h_{12}(u))>0,$ then $E$ does not have multiplicative reduction at $p$, since $\res \left(h_{12}(u), g_{\Delta}(u)\right)=\res \left(h_{12}(u), g_{c_4}(u)\right)=1,$ where $g_{\Delta}$ is the numerator of $\Delta (u)$ and $g_{c_4}$ is the numerator of $c_4 (u),$ and therefore $v_p(j(u))=0.$ There are several possibilities for the condition $v_p(\Delta(u))>0$ to be satisifed:
\begin{itemize}
	\item If $v_p(u) \eqqcolon k>0$, then using the fact that $$\res\left(u, \frac{\Delta(u)}{u^{14}}\right)=\res\left(u,c_4(u) \right)=1,$$ we conclude that reduction mod $p$ will be of type $I_{14k}$.
	
	\item If $v_p(u) \eqqcolon -k<0$, with the change of variables $v\coloneqq \frac{1}{u}$ we get a similar situation as above, with $$\res\left(v, \frac{\Delta(v)}{v^{14}}\right)=\res\left(v,c_4(v) \right)=1,$$ so the reduction mod $p$ will be of type $I_{14k}$.

	\item If $v_p(u)=0$ and $v_p(u+1)\coloneqq k>0,$ then using the fact that $$\res\left(u+1, \frac{\Delta(u)}{(u+1)^{14}}\right)=\res\left(u+1,c_4(u) \right)=1,$$ we conclude that the reduction mod $p$ is of type $I_{14k}.$
	\item The only other possibility for multiplicative reduction is $v_p(u^3 + u^2 - 2u - 1)\eqqcolon k>0$. Note that a root $\alpha$ of $f(u) := u^3 + u^2 - 2u - 1$ generates the ring of integers $\Z[\alpha]$ of $\Q(\zeta_7)^+$. The fact that $p|f(u)$ implies that $f(u)$ has a root in $\F_p$ and hence $p$ is not inert $\Q(\zeta_7)^+$, implying $p\equiv \pm 1 \pmod 7$ or $p=7$. Since
	$$\res\left(u^3 + u^2 - 2u - 1, \frac{\Delta(u)}{(u^3 + u^2 - 2u - 1)^{2}}\right)=7^{30}$$ and $$\res\left(u^3 + u^2 - 2u - 1, c_4(u)\right)=7^{12},$$ it follows that there can be cancellation with the numerator only in the case $p=7$.
	\item Suppose $p=7$, $v_7(u)=0$ and $v_7(u^3 + u^2 - 2u - 1)=k>0$. An easy computation shows that $u \equiv 2 \pmod 7$ and $k=1$, and that the numerator will be divisible by a higher power of $7$ than $u^3 + u^2 - 2u - 1$, which show that the reduction will not be multiplicative.

\end{itemize}
\end{proof}

\begin{proposition}
The curve $E$ has multiplicative reduction of type $I_{14k}$ at $2$.
\end{proposition}
\begin{proof}
	This follows from the observation that $v_2(u)\neq 0$ or both $v_2(u)=0$ and $v_2(u+1)>0$, from which it follows, by the calculations in the proof of \Cref{j-inv}, that in both cases the reduction type of $E_u$ at $2$ is $I_{14k}$.
\end{proof}


We now prove four useful lemmas.

\begin{lemma}\label{lem:fd}
	Let $x\in Y(K)$ and let $\PP$ be a prime of $K$ over $2$. Then $x$ modulo $\PP$ is defined over $\F_2$.
\end{lemma}
\begin{proof}
As mentioned above, the results of \cite{BN2} imply that a non-cuspidal cubic point on $x\in X$ given by the equation $f(u,v)=0$ in \eqref{X214}
satisfies either $u\in \P^1(\Q)$ or $v\in \P^1(\Q)$. Over $\F_2$, the polynomial $f$ factors as
$$f(u,v)=(u+v)(uv+v+1)(uv+v+1),$$
which implies that if one of $u$ or $v$ is $\in \P^1(\F_2)$, then so is the other. This implies that the reduction of $x$ modulo $\PP$ is defined over $\F_2$.
\end{proof}

\begin{lemma}
	\label{lem:1}
	Let $F=\Q(\zeta_7)^+$, let $C$ be a cusp of $X$ whose field of definition is $F$ and let $q$ be a rational prime. Then the field of definition of the reduction of $C$ in $\overline \F_q$ is $\F_{q^3}$ if $q\not\equiv \pm 1 \pmod 7$ and $\F_q$ if $q\equiv \pm 1 \pmod 7$.
\end{lemma}
\begin{proof}
	We have $[k(C):\F_q]=[\Q_q(\zeta_7+\zeta_7^{-1}):\Q_q]$ from which the claim follows.
\end{proof}

\begin{lemma}\label{pm1}
  Let $q\equiv \pm 1 \pmod 7$ be a rational prime such that $E$ has multiplicative reduction over $q$ and let $\PP$ be a prime of $K$ over $q$. Then the field of definition of the reduction of $x\in X$ modulo $\PP$ corresponding to the curve $E$ is $\F_q$.

\end{lemma}
\begin{proof}
  Since $x$ modulo $\PP$ is a cusp, the statement follows from \Cref{lem:1}.
\end{proof}

\begin{lemma}\label{lem:reduction}
Let $q$ be a rational prime and let $\fq$ be a prime of
$K':=K\Q(\zeta_7)^+$ above $q$. Then the reductions of the cusps of $X_1(14)(K')$ modulo $\fq$ are all distinct.
\end{lemma}
\begin{proof}
Let $z$ be a cusp of $X_1(14)$. Since the point at infinity $\mathcal O$ of the elliptic curve $X_1(14)$ is a rational point and all the rational points on $X_1(14)$ are cusps, it follows that $z=z-\mathcal O$ (viewed as an element of the Jacobian of $X_1(14)$) is a torsion point by the Manin-Drinfeld theorem \cite{manin,drinfeld}. Since all the cusps of $X_1(14)$ are contained in $X_1(14)(\Q(\zeta_7)^+)$, it is enough to prove that the reductions modulo $\fq'$, where $\fq'$ is a prime of $\Q(\zeta_7)^+$ below $\fq$, of the cusps, viewed as elements of $X_1(14)(\Q(\zeta_7)^+)$, are all distinct.

For $q>2$ reduction modulo $\fq'$ is injective on the torsion of $X_1(14)(\Q(\zeta_7)^+)$ by \cite[Appendix]{katz} and for $q=2$ we explicitly check that reduction modulo $\fq'$ is injective.
\end{proof}

\begin{proof}[Proof of \Cref{thm:main_sec3}]
	Let $\sigma$ be a generator of $\Gal(K/\Q)$ and suppose $q$ is inert or ramified in $K$ and $\PP$ is the unique prime over $q$. Let $x\in Y(K)$ be a point representing the quadruple $(E,P,Q,R)$ as defined at the beginning of this section. As the degree 3 map $X\rightarrow\P^1$ is quotienting by $\alpha$, it follows that
$$\left\{ x, x^\sigma, x^{\sigma^2}\right\}= \left\{ x, \alpha(x), \alpha^2(x)\right\},$$
so we can suppose without loss of generality that $x^\sigma=\alpha(x)$ and $x^{\sigma ^2}=\alpha^2(x)$. Let $\overline x=\overline {C_0}$, for some cusp $C_0\in X$. It follows that $\overline{\alpha(x)}=\overline{\alpha(C_0)}$ and $\overline{\alpha^2(x)}=\overline{\alpha^2(C_0)}$. Denote by $C_1:=\alpha(C_0)$ and by $C_2:=\alpha^2(C_0)$; all $C_i$ are distinct as $\alpha$ acts without fixed points on the cusps. By \Cref{lem:fd} and \Cref{pm1}, all $\overline C_i$ are defined over $\F_q$.

Denote by $K_i:=\phi(C_i)$ and by $y=\phi(x)\in Y_1(14)$. Descending everything to the elliptic curve $X_1(14)$, we have $\overline y=\overline {K_0}$, $\overline{y^{\sigma}}=\overline{{K_1}}$, $\overline{ y^{\sigma^2}}=\overline {K_2}$.  
	By \Cref{lem:fd} and \Cref{pm1}, all $\overline {C_i}$ and hence all $\overline K_i$ are defined over $\F_q$.

Denote by $\Frob \PP$ the Frobenius substitution at $q$ in $\Gal(K/\Q)$ if $q$ is inert in $K$, and put $\Frob \PP:=id $ if $q$ ramifies in $K$. Suppose without loss of generality that $\Frob \PP=\sigma$ if $q$ is inert. Following the same reasoning as in \cite[Proposition 3.1]{najman} and using the fact that $\overline y\in X(\F_q)$, we get that $$\overline{K_0}=\overline{y}=\overline{y^{\Frob \PP}}= \overline{y^\sigma}= \overline{K_1},$$
and similarly $\overline{K_0}=\overline{K_2}$.

By \Cref{lem:reduction} it follows that $K_0=K_1=K_2$. This is impossible since $C_0,C_1,C_2$ are distinct and $\phi$ is a degree $2$ map.
\end{proof}

\subsection*{Acknowledgments}
We thank the referees and Andrej Dujella for many helpful comments that greatly improved the exposition of the paper.

\bibliographystyle{unsrt}

\end{document}